\documentclass[12pt]{article}%
\usepackage[applemac]{inputenc}
\usepackage{amsmath,amssymb}
\usepackage{amsfonts}
\usepackage{amsmath,amssymb}
\usepackage[applemac]{inputenc}
\usepackage{color}
\usepackage{color, colortbl}
\usepackage{amsmath}
\usepackage{amssymb}
\usepackage{bbm}
\usepackage{graphicx}
\usepackage{tikz}
\usepackage{geometry}
\usetikzlibrary{arrows}
\usepackage{amsfonts}
\usepackage{amsmath,amssymb}
\usepackage[applemac]{inputenc}
\usepackage{color}
\usepackage{amsmath}
\usepackage{amssymb}
\usepackage{graphicx}
\usepackage{tikz}
\usepackage{bbm}
\usepackage{mathtools}
\usepackage{amssymb}
\usepackage{comment}
\usepackage{float}
\usepackage{amsthm}
\newtheorem{theorem}{Theorem}[section]
\newtheorem{lemma}[theorem]{Lemma}

\newtheorem{problem}[theorem]{Problem}
\newtheorem{remark}[theorem]{Remark}

\def\R{\mathbb{R}}

\def\cE{\mathcal{E}}

\def\cA{\mathcal{A}}
\def\cF{\mathcal{F}}

\def\bE{\mathbb{E}}
\def\bP{\mathbb{P}}
\def\bR{\mathbb{R}}
\def\bN{\mathbb{N}}

\begin{document}

\title{Installation of renewable capacities to meet emission targets and demand under uncertainty}
\author{Nacira Agram$^{1,2},$ Fred Espen Benth $^{3}$, Giulia Pucci$^{1}$ }
\date{\today}
\maketitle

\begin{abstract}
This paper focuses on minimizing the costs related to renewable energy installations under emission constraints. We tackle the problem in three different cases. Assuming intervening once, we determine the optimal time to install and the optimal capacity expansions under uncertainty. By assuming the possibility of two or multiple interventions, we find that the optimal strategy is to intervene only once. We also prove that there are instances where capacity expansions should be delayed.
\end{abstract}

\footnotetext[1]{Department of Mathematics, KTH Royal Institute of Technology 100 44, Stockholm, Sweden. \newline
Email: nacira@kth.se, pucci@kth.se. 
}

\footnotetext[2]{ Work supported by the Swedish Research Council grant (2020-04697) and the Slovenian Research and
Innovation Agency, research core funding No.P1-0448.
}

\footnotetext[3]{
Department of Mathematics, University of Oslo, PO Box 1053 Blindern, 0316 Oslo, Norway. 
Email: fredb@math.uio.no.}


\section{Introduction}
Our main focus in the current paper is to study under uncertainty how to optimally install renewable energy sources such as solar, wind, hydro, or geothermal power plants covering energy demand.
By installing renewable capacity,  which generates power stochastically, we aim to meet the randomly varying energy demand of consumers while reducing reliance on fossil fuels and thereby decreasing greenhouse gas emissions. 

 To fullfil a given energy demand, we have the flexibility to either utilize existing fossil energy generation, or to (partly) substitute with installing renewable energy sources, reflecting the current energy landscape. It is important to note that, while fossil energy production can provide a stable supply of energy, it also comes with the drawback of generating emissions, such as carbon dioxide and other  greehouse gas pollutants. On the other hand, renewable energy production offers a cleaner and sustainable option for meeting energy needs but requires substantial investments in infrastructure and technology. 
In particular, in our study we will take into consideration a fixed initial investment cost required for the installation, and proportional costs that vary depending on the capacity size of the entity installed. 

Our objective is twofold: to minimize the costs associated with the installations of renewable energy, while ensuring that we meet our emissions reduction targets, particularly in the context of fossil energy usage. In this regard, we will consider the total accumulated emissions, which originate from energy produced using fossil sources and from energy imported from abroad. The emission target can be seen as a financial penalty added to the cost function that needs to be minimized or as a probability constraint to be verified. Our model seeks to find the optimal balance between cost minimization and environmental responsibility. By employing this approach, we contribute to the broader goal of achieving sustainable, cost-effective, and environmentally friendly energy systems, in line with the global shift towards cleaner energy sources and reduced carbon emissions.


We are constantly faced with decisions, whether they are our own or made by others. When contemplating decisions in mathematics, we employ the theory of optimal control. Given the prevalence of uncertain information in many of our decisions, we must address optimal control under uncertainty. Furthermore, the existence of feasible regions and other real-world requirements, necessitates our involvement in optimization under uncertainty with constraints.
Optimal control techniques under uncertainty empower decision-makers to devise robust strategies for minimizing the installation of renewable energy sources while adhering to emission constraints. By formulating mathematical models that integrate uncertain parameters, such as future energy demand projections and variability in renewable energy generation, decision-makers can optimize the deployment of renewables over time. 

In Dentcheva \cite{D} optimization problems are tackled where constraints involve uncertainty described by probability distributions (see also Calafiore and Dabbene \cite{CD}).
There are also some articles which explore various aspects of energy system optimization considering observed weather data and uncertainty, ranging from demand response to microgrid operation and renewable energy integration. They provide valuable insights into how weather-driven optimization can improve the efficiency, reliability, and sustainability of energy systems. Without being exhaustive, we refer to D{\'\i}az-Gonz{\'a}lez \textit{et al.} \cite{DSBR}, Mohseni \textit{et al.} \cite{MBKB}, Pfeifer \textit{et al.} \cite{PDPKD} and Ren \textit{et al.} \cite{RJWLL}. We also refer to A\"id {\it et al.} \cite{ADT}, Bassi\'ere {\it et al.} \cite{BDT} and Dumitrescu {\it et al.} \cite{DLT}, which examine Nash equilibria in capacity expansion problems within a mean field games framework involving renewable and conventional energy producers.

  In another stream of research, Zeyringer {\it et al.} \cite{Zeyr} analyse cost-minimizing capacity expansion problems of renewable power systems for the UK, where uncertainty is analysed by variability in solutions across different weather-years. A weather-year is providing information about the production of renewables (through wind speed and solar irradiation) and temperatures (driving the demand). In a more recent paper, by relaxing cost-optimality, Grochowicz {\it et al.} \cite{Groch} extracted feasible sets of capacity expansions (near-optimal capacities) showing the flexibility to account for robustness towards variations in production described by different weather-years. The complex interaction between extreme weather variability, power production and the energy system is studied in Grochowicz {\it et al.} \cite{GrochBloom}. These studies are not using any probability distribution to
span the possible scenarios weighted by their probabilities, but one is rather (sub-)optimizing the system for each weather-year. As we do not know the future but may at best only assess the stochastic scenarios, we must make decisions based on this uncertainty rather than fixing the scenarios and then making the decision.

Our main findings reveal that under uncertainty, it may be worthwhile to postpone green investments. This contrasts with the findings of Victoria \textit{et al.} \cite{VBG}. We also discover that it is not optimal to invest multiple times. Rather, one should opt for a single comprehensive investment. This perspective may differ from the approach many authorities plan for, which involves stepwise installation of renewables.  An example here is the Norwegian Government's plan for offshore wind capacity expansion, where the first
step was taken in 2024 with the auctioning of installation rights in S\o rlige Nordsj\o, and with further expansions planned for the future in this and other offshore areas.

The organization of the paper is as follows: Section 2 is devoted to the model formulation, where we provide all the necessary ingredients to describe the optimization problems. In Section 3, we study the one intervention case and illustrate the results obtained by solving a numerical example. Sections 4 and 5 are devoted to the twice- and multiple-interventions cases, respectively.

\section{Model formulation}\label{problem}
In this section, we introduce the optimization problem whose objective is twofold: to minimize the costs associated with the installations of renewable energy, while ensuring that we meet our emissions reduction targets from fossil energy usage.  A certain fixed amount of regulated power such as gas and nuclear is assumed to be installed in the system.

Assume there exist $k$ different technologies of variable energy production, including renewable options like hydro, onshore and offshore wind, and solar power. Each of these technologies can be installed in $n_i\in\mathbb N_0$, $i=1,\ldots, k$, different places within a specified geographic area. Let $d = n_1 + \ldots + n_k$ be the total dimension, and $V(t)\in [0,1]^d \in \bR^d$ be a vector of capacity factors at time $t\ge 0$. We denote by $D(t) \ge 0$ the energy demand (measured in MWh) at time $t$. We suppose that the capacity factor $V$ and the energy demand $D$ are stochastic processes defined on a given filtered probability space $(\Omega, \cF, \{ \cF\}_{t \ge 0}, \bP)$, which are $\cF_t$ predictable.  Let furthermore $T$ be a fixed terminal time horizon (e.g. year $2050$, say).

Let $t  \mapsto C_R(t)$ be an $\R^d$-valued function representing the installed capacities of given renewable technologies in specific locations. $C_R(t)$ is piece-wise constant in each coordinate with jumps occurring at deterministic times. These jumps are non-negative and non-decreasing, indicating planned capacity expansions. Our control consists in a sequence of capacity expansions $\xi_1,\xi_2,\ldots $ at fixed times $t_1 < t_2 < \ldots < T $ such that
\begin{equation*}
    C_R(t) = C_R(0^-) + \sum_{i\colon t_i \le t} \xi_i.
\end{equation*}
Here, $C_R(0^-)$ is the initial capacity and we suppose that $\xi_i \in \bR^d_+, t_i > 0$. We let $v = (t_1,t_2,\ldots; \xi_1,\xi_2,\ldots)$ be a decision variable that we can control. 

The proportional cost of installation per unit MW of \textit{variable }capacity  is measured by a given $\bR^d_+-$valued function $K_v(t)$. Thus, we allow the proportional costs to vary with time. Naturally, the costs should \textit{decrease} with time, as we measure everything at present time and there is a discounting that must be accounted for. Also, as technologies mature, one expects cheaper PV (photovoltaic) panels and windmills in nominal terms as well. The canonical choice here is pure discounting, $K_v(t) = \kappa_v \exp (- \rho t) $ for a discount rate $\rho > 0$ and constant $\kappa_v \in \bR^d_+$. 
The total accumulated cost up to time $t$, which we aim to optimize, is as follows  (using the notation $x^{\intercal}$ for the transpose of a column vector $x$)
\begin{equation*}
    \sum_{i\colon t_i \le t} K_f(t_i)^\intercal  \mathbbm{1}(\xi_i>0) + K_v(t_i)^\intercal \xi_i,
\end{equation*}
where $K_f \in \bR^d_+$ is the fixed cost for starting a capacity expansion (which can vary with location as well as technology, and can be thought of as the cost of setting up infrastructure etc...). These costs are also naturally changing with time, having pure discounting as the canonical choice. We denote by
$\mathbbm{1}(\xi > 0)$ the $d$-dimensional vector which is $1$ in coordinate $j$ when $e_j^\intercal \xi >0 $ ($e_j$ is the $j$th basis vector in $\R^d$) and zero otherwise (i.e., a vector which indicates with $1$ the places and technologies for which we do an expansion, and zero for the coordinates where there is no expansion). 

We have technologies for regulated power production as well, including gas, coal and nuclear. We denote by $C_F \in \bR_+$ (denominated in MWh) the maximal possible power production from these collectively, where the subscript $F$ signifies "fossil" as there is a proportional emission rate coming from these regulated power technologies. We suppose $C_F$ is fixed and given. 

The variable power production at time $t$, which depends on the capacity factors and the installed capacities in each location, is defined as 
\begin{equation*}
    P_R(t) \coloneqq C_R(t)^\intercal V(t),
\end{equation*}
while the regulated power production is 
\begin{equation}
\label{eq:regulated-production}
  \begin{aligned}
        P_F(t) = P_F(D(t)-P_R(t)) \coloneqq \begin{cases}
        0, \quad & \text{if } D(t) - P_R(t) \le 0, \\ 
        D(t) - P_R(t) \quad & \text{if } 0 < D(t) - P_R(t) \le C_F, \\ 
        C_F \quad &\text{if } D(t) - P_R(t) > C_F.
    \end{cases}
  \end{aligned}
\end{equation}
 In this setup for production, renewable power generation is given priority above regulated power production. This is the case in the German electricity system, say, to
incentive clean production.
Note that Equation \eqref{eq:regulated-production} accounts for various scenarios based on the interplay between energy demand $D(t)$ and renewable energy production $P_R(t)$.
In the first case, the demand for energy is entirely met by renewable energy production, and thus there is no need to generate energy from fossil sources, resulting in a power production $P_F$ of $0$.
In the second case, we will only produce the amount of fossil energy needed to bridge the gap between demand and renewable energy production.
Finally, the last case arises when the energy demand surpasses the maximum power production capacity from fossil sources and, in this situation, we will generate power equal to the maximum capacity, which may fall short of fulfilling the entire energy demand. In such instances, additional energy import is required to meet the remaining demand.
We assume here that whenever we produce more renewable $P_R$ than we need, e.g. $D(t) \le P_R(t)$, we can export this, while if we cannot meet demand with renewable and regulated power, e.g., $D(t) \ge P_R(t) + C_F$, we can import what is needed. Hence,  in this general setup we suppose that our power market is connected with external markets where there is always sufficient power to cover a deficit in production and always a possibility to off-load surplus production.

Fossil-fueled power production incurs emissions domestically. Let $e_F$ be the emission rate from $P_F(t)$, and define the total accumulated emissions at time $t$ to be:
\begin{equation} \label{1.4}
    E(t) \coloneqq \int_0^t e_FP_F(D(s)-P_R(s))ds.
\end{equation}
Notice that $E(t)$ implicitly depends on the installed capacity $C_R(s), s \le t$, through $P_F$. Alternatively, we can include in the accumulated emissions the emissions originating from imported power. For instance, Germany imports power from Poland, which is  to a large extent generated from coal. However, these emissions are allocated to Poland's emission budget  according to international regulations, and thus, Germany is not directly responsible for them in its power generation. This implies that Germany can offload some of its emissions simply by importing power from abroad, without being concerned about the emissions associated with it. To prevent such "export of emissions", one could impose a penalty for having to import power. Such a penalty can be seen as either a political measure to ensure energy self-sufficiency or as a response to public pressure on policymakers, as was the case in Norway where electricity bills were soaring in  2023 raising public concerns about "high price imports" from continental Europe. The formula for the total accumulated emissions takes the following form:
\begin{equation}
\label{1.5}
    E(t) \coloneqq \int_0^t \cE (D(s)-P_R(s))ds,
\end{equation}
where 
\begin{equation*}
  \begin{aligned}
        \cE(x) \coloneqq \begin{cases}
        e_F x \quad & \text{if } 0 \le x < C_F, \\ 
        0 \quad & \text{if } x<0, \\ 
       e_I(x-C_F) + e_F C_F \quad &\text{if } x \ge C_F,
    \end{cases}
  \end{aligned}
\end{equation*}
and $e_I$ is the import emission factor. Notice that when $D(t)-P_R(t) < 0$, we have in any case that $P_F(t)=0$ which is the case $x<0$ above. This emission rate from imports can be interpreted literally as the emissions, or as a way to penalize political risk from imports as discussed above. For simplicity, we will
 in this paper set $e_I = 0,$ and work with the emissions as defined in \eqref{1.4}.

Set $\cA$ to be the set of admissible controls $C_R$, meaning the set of all sequences $v= (t_1,t_2,\ldots; \xi_1, \xi_2,\ldots)$.
We consider the cost-minimization problem
\begin{problem}
\label{problem1}
\begin{equation*}
    \inf_{C_R \in \cA} J(C_R),
\end{equation*}
where 
\begin{equation*}
    J(C_R(t)) \coloneqq  \bE \sum_{i \colon t_i \le T} \left( K_f(t_i)^\intercal \mathbbm{1}(\xi_i> 0) + K_v(t_i)^\intercal \xi_i\right) + \int_0^{T}f(C_R(t))dt + g(E(T)).
\end{equation*}
\end{problem}
Here, $f \colon \bR^d_+ \to \bR_+$ can be thought of as an operational cost, i.e. the maintenance cost for the conservation of the plants and it depends on the installations we 
made, 
while  $g \colon \bR_+ \to \bR_+$ is penalizing the increase of accumulated emissions. 
We impose the following set of assumptions:
\begin{itemize}
    \item [(i)]There exists a constant $K'>0$, such that
 $f(C_R(0))<K'$. 
    \item [(ii)] $g$ is a non-decreasing function of at most polynomial growth of order $k \in \bN$. 
    \item [(iii)] $E(T)$ has finite $k-$moment (for the same $k$ as in (ii) above).
\end{itemize}
\begin{lemma}
\label{lemma2.2}
    Under the assumptions (i)-(iii), the Cost Minimization Problem \ref{problem1} is well-defined, meaning $$0 \le \inf_{C_R \in \cA}J(C_R) < \infty.$$
\end{lemma}
\begin{proof}
    The do-nothing approach is feasible: Let $C_R(s) \equiv C_R(0)$, which is an admissible process. Then $$ J(C_R(0))=\bE \int_0^T f(C_R(0))dt + g(E(T)) = \bE f(C_R(0))T + g(E(T)).$$ Under the above assumption (ii) on $g$, we have for all $x \in \bR_+$ there exists a $K > 0$ such that
\begin{equation*}
    g(x) \le K(1+x^k).
\end{equation*}
Therefore, from the assumption (i) on $f$,  $$J(C_R(0)) \le K'T+K(1+\bE[E(T)^k]),$$ which is finite by assumption (iii). 
\end{proof}
We can weaken the moment assumption on $E$. By the definition \eqref{1.4} and using H\"older's inequality, we have that 
\begin{equation*}
    \left ( \int_0^T \cE(D(s)-P_R(s))ds \right )^k \le T^{k-1}\int_0^T \cE(D(s)-P_R(s))^kds.
\end{equation*}
We recall that $\cE(s)$ is at most linearly growing in $x$. Moreover, when $C_R(t)=C_R(0)$, we have $P_R(s)=C_R(0)^\intercal V(s)$. Notice that since $V(t) \in [0,1]^d,$ all the moments are trivially finite (in fact bounded by $1$). So, if we assume a finite, integrable over $[0,T]$, $k$th moment of $D(s)$, we have the desirable $k$th moment of $E(T)$.

In the general formulation, we use a penalty function 
$g$. However, a penalty function 
$g$ could aim to have a target 
$L$ for the accumulated emissions and require the installation of renewables to meet the probabilistic target.
\begin{equation}
    \bP(E(T) \ge L)=\varepsilon.
    \label{prob_constr}
\end{equation}
This means that one installs sufficient renewable capacity so that the total emissions violate the target by only an assigned small threshold probability $\varepsilon>0$. We could formulate the target as a probability less than or equal to $\varepsilon$, but it is more cost-optimal to go for maximum probability within the range allowed. The emissions are linearly increasing in $D(s)-P_R(s)$ so to reduce the probability we must increase the capacity of renewables. But this increases expenses, so we are better off by maximizing the possible probability, i.e., let it be equal to $\varepsilon$. 

By choosing the probabilistic constraint described in \eqref{prob_constr}, we get the following optimization problem under probabilistic constraints
\begin{problem}
    \label{problem2}
    \begin{equation*}
    \inf_{C_R \in \cA} J(C_R) \quad \textnormal{s.t.} \quad \mathbb{P}(E(T) \ge L ) = \varepsilon,
\end{equation*}
where 
\begin{equation*}
    J(C_R) \coloneqq \bE \sum_{i \colon t_i \le T} \left( K_f(t_i)^\intercal \mathbbm{1}(\xi_i> 0) + K_v(t_i)^\intercal \xi_i\right) + \int_0^{T}f(C_R(t))dt.
\end{equation*}
\end{problem}

\begin{lemma}  
\label{lemma2.4}  For every $\varepsilon \in [0,1]$ and for $L>0$ such that \begin{equation}
    \bP\left(e_F\int_0^T\min( D(s),C_F)ds \ge L \right) >> \varepsilon,
    \label{cond2}
\end{equation}  the cost minimization Problem \ref{problem2} is well-defined.
\end{lemma}
\begin{proof} 
The proof relies on standard probability tools.
Heuristically, we can install initially sufficient capacity to honour the probabilistic constraint \eqref{prob_constr} and show that the problem is feasible, so we just need to prove that there exists a $\xi>0$ such that 
\small
\begin{equation*}
    \bP\left [ \int_{0}^T (D(s)-\xi V(s))\mathbbm{1}_{\{s \colon D(s) - \xi V(s) \in [0,C_F]\}} ds +  \int_0^{T} C_F \mathbbm{1}_{\{s \colon D(s) - \xi V(s) > C_F \}} ds \ge \frac{L}{e_F} \right] =\varepsilon.
\end{equation*} 
Without loss of generality we supposed that the capacity installed at time zero $C_R(0) = 0$.  
By naming 
\begin{equation*}
    I(\xi) \coloneqq \int_{0}^T (D(s)-\xi V(s))\mathbbm{1}_{\{s \colon D(s) - \xi V(s) \in [0,C_F]\}} ds +  \int_0^{T} C_F \mathbbm{1}_{\{s \colon D(s) - \xi V(s) > C_F \}} ds, 
\end{equation*}
we notice that $I$ is continuous with respect to $\xi$. Thus, by proving that the function $\bP\left(I(\xi) \ge \frac{L}{e_F} \right) $ covers a neighborhood of $\varepsilon$, we can use the Intermediate Value Theorem to show that there exists a feasible $\xi$. This is straightforward, because if $\xi \to 0$, an upper bound comes from condition \eqref{cond2}, while sending $\xi \to \infty$ results in $I(\xi) \to 0$, thus having $\bP\left( I(\xi) \ge \frac{L}{e_F}\right) \to 0.$
\end{proof}
\begin{remark}
    Condition \eqref{cond2} tells us that, when setting the threshold parameters 
$\varepsilon$ and $L$ for the probability constraint, it is important to choose realistic and achievable values. If, for example, the threshold 
$L$ is set too high, i.e. it far exceeds any feasible amount of tons of $CO_2$
  that could be managed within the given energy demand over time, the probability of meeting this threshold will be zero and the probability constraint in \eqref{prob_constr} loses its significance. On the other hand, \eqref{cond2} is also telling us that the do-nothing approach is not feasibile, and we need a positive capacity expansion to meet the constraint. 
\end{remark}

\begin{remark}
    We can re-formulate the probability constraint in \eqref{prob_constr} as \newline $\bE[\phi(E(T))] = 0$ with $\phi= \mathbbm{1}(x > L) - \varepsilon$. Introducing a Lagrange multiplier, we obtain the optimization problem
\begin{equation*}
    \inf_{C_R \in \cA}
 \coloneqq  \bE \sum_{i \colon t_i \le T} \left( K_f(t_i)^\intercal \mathbbm{1}(\xi> 0) + K_v(t_i)^\intercal \xi_i\right) + \lambda\phi(E(T)).
\end{equation*}
For positive $\lambda$'s $g(x) \coloneqq \lambda \phi(x)$ becomes a non-decreasing function in $x$ being bounded by $\lambda$. 
\end{remark}


Considering different technologies such as solar, wind, hydro, etc. in different locations can provide a detailed understanding of the performance and cost-effectiveness of each technology in each location. At the same time, a significant amount of information is needed to accurately model and compare the different technologies across the space, like data on resource availability, installation and maintenance costs, efficiency, and environmental impact. A natural alternative is to aggregate all the renewable energy sources in a single category and choose $d=1$.  It is important to note that, while individual renewable technologies have distinct cost structures, when combined, we must refer to average characteristics and costs, which can serve as a useful proxy for the overall investment required in the renewable sector. In the remaining of the paper we consider the situation
where we have aggregated all renewables to have $d=1$.  Our main focus in the sequel of the paper will be on Problem \ref{problem2}, the cost-minimization with 
probabilistic constraint. We will include additional restrictions on the admissible capacities as well in our analysis.

\section{Case I: One intervention}

%
In this section, we will solve a constrained minimization problem in which a single installation at a deterministic time is possible.
The aim is to minimize installation costs while meeting a constraint on emissions. 

Consider a situation where we can intervene once at time $0 \le t \le T $ with capacity expansions $\xi \in (0,\xi_M]$ for some $\xi_M \in \bR$  maximum capacity to be installed.  We suppose that the initial installed capacity is $C_R(0^-)=0$.  Choosing a discounting of the proportional costs and no fixed costs, the constrained cost minimization problem is as follows: 
\begin{equation}
    \min_{0\le t \le  T, \xi \in (0,\xi_M]} \kappa \xi e^{-\rho t} \qquad \textnormal{subject to } \qquad \mathbb{P}(E(T) > L ) = \varepsilon.
    \label{1intervention}
\end{equation}
To define the accumulated emissions, we first need to define some variables.

The power production at time $s \in [0,T]$ is given by: 
\begin{equation*}
    P_R(s) \coloneqq C_R(s)V(s)= V(s) \xi\mathbbm{1}{[s \ge t]}. 
\end{equation*} 
This will give
\begin{equation*}
    D(s)-P_R(s) \coloneqq \begin{cases}
  D(s) & \textnormal{if } s < t,  \\
       D(s)-\xi V(s)  & \textnormal{if } s \ge t.
\end{cases}
\end{equation*}
For simplicity, we will  suppose 
\begin{equation*}
    D(s) - P_R(s) \le C_F,
\end{equation*}
which means that, in order to meet energy demand, one does not need to import energy from abroad. In particular, we have that $D(s)\leq C_F$ when $s\leq t$, i.e., we have reserve capacity in fossil power generation to fully cover demand, at least up to some time. Further, the demand is bounded.

Therefore, the total accumulated emissions at time $t$ are given by
\begin{equation}\label{AE}
    E(T) \coloneqq E(t,\xi,T) = e_F \int_{t}^T (D(s)-\xi V(s))^+ ds + e_F \int_0^{t} D(s)ds,
\end{equation}
with corresponding partial derivatives 
\begin{equation*}
   \begin{aligned}
        &\partial_t E(t,\xi,T) = e_F D(t) - e_F(D(t)-\xi V(t))^+ =e_F \min(D(t),\xi V(t)) > 0, \\
        &\partial_\xi E(t,\xi,T)  = - e_F \int_t^T V(s)\mathbbm{1}(D(s)- \xi V(s)>0)ds \le 0.
   \end{aligned}
\end{equation*}
 Here we use the notation $\partial_x=\partial/\partial x$ for partial differentiation with respect to a variable $x$.
Notice that the cost function $t \mapsto \exp(-\rho t) \kappa \xi$ is decreasing, while $ t \mapsto E(t,\xi, T)$ is increasing with fixed $\xi$. Hence, to minimize the costs, it is reasonable to intervene as late as possible, but this must be balanced against intervening sufficiently early to reach the emission target. In particular, our claim is that the optimal intervention time corresponds to a trade-off between intervening  sufficiently early to reach the emission target and sufficiently late to reduce costs.

 In fact, if the intervention time is set to $t=T$, then the accumulated emissions up to the final time is given by $E(T)=e_F\int_0^T D(s)ds$. But $L$ and $\varepsilon$ should naturally be set in a practical situation to be such that
\begin{equation}
    \bP\left(e_F\int_0^T D(s)ds > L \right) >> \varepsilon.
    \label{cond}
\end{equation}
In this situation, we do not introduce any renewable power production, and only rely on fossil fuels (domestic and imported). While undoubtedly effective in cost reduction, this strategy is highly likely to fail in meeting emission constraints. Therefore, the optimal intervention time should be $ \hat{t}<T$. 

To continue the analysis, we make the further simplification that demand and capacity factor are both time-independent random variables, i.e., $D$ and $V$. Furthermore, we assume that
\begin{equation}
    \bP\left(D < \xi_M V \right)\approx 0,
    \label{DV}
\end{equation} i.e., renewable energy generation will not be able to meet all of the demand even at maximal possible installation. This is a natural assumption given that 
wind and solar power is depending on weather, and there may be extended periods of very low production with simultaneous high demand (heat waves combined with
"Dunkelflaute", say).
The total accumulated emission \eqref{AE} takes the form
\begin{equation*}
    E(t,\xi,T)=e_F(T-t)(D-\xi V)^+ + te_FD \approx e_FTD -e_F(T-t)\xi V.
\end{equation*}

\begin{remark}
We may also, roughly speaking, {\it define} the variables $V$ and $D$ by the relations $V:=\int_t^TV(s)ds/T-t$ and $D:=\int_t^TD(s)ds/T-t$. Obviously, this is at stake with the
fact that $t$ is not fixed, but provides a simplification that is useful in getting insight into the capacity expansion problem. 
\end{remark}

Condition \eqref{DV} is expressing that it is basically impossible to cover all demand by renewables only (or, the probability of this happening is negligible). The emission constraint in \eqref{1intervention} then becomes
\begin{equation*}
    \bP \left( \frac{TD-\frac{L}{e_F}}{V}>\xi(T-t) \right ) = \varepsilon. 
\end{equation*}
Hence, $z_\varepsilon \coloneqq \xi \times (T-t)$ is the $\varepsilon$-upper quantile of the random variable $(TD-(L/e_F))/V$. We assume $z_\varepsilon$ to be strictly positive and if
 \begin{equation}
    0 < \frac{z_\varepsilon}{T}     \le \xi_M,
    \label{quantile}
\end{equation}
 we will have an admissible unique capacity expansion 
\begin{equation}
    \xi_\varepsilon \coloneqq \frac{z_\varepsilon}{T-t}.
    \label{xi}
\end{equation}
 At this point we have all the necessary tools to determine the optimal time to intervene.

\begin{theorem}
\label{1int}
    Consider a suitable target set for the accumulated emissions, denoted as $L > 0$, along with a threshold value $\varepsilon > 0$ as in \eqref{cond}. 
    Moreover, we assume that both demand $D$ and capacity factor $V$ are time-independent random variables satisfying \eqref{DV} and such that the $\varepsilon-$upper quantile $z_\varepsilon$ satisfies \eqref{quantile}.
    
    According to the discounting rate $\rho$, we distinguish between three cases:
    \begin{itemize}
        \item [1.]    If $\frac{1}{T} \le \rho \le \frac{\xi_M}{z_\varepsilon}$,      
 then the optimal solution to the installation problem \eqref{1intervention} is uniquely determined by 
    \begin{equation}
         \hat{t}= T-\frac{1}{\rho}, \qquad \hat{\xi}_\varepsilon \coloneqq \rho z_\varepsilon.
          \label{1}
    \end{equation}
    \item [2.] If $ \rho < \frac{1}{T}$, then the optimal solution to \eqref{1intervention} is given by 
    \begin{equation}
          \hat{t}= 0, \qquad \hat{\xi}_\varepsilon \coloneqq  \frac{z_\varepsilon}{T}.
          \label{2}
    \end{equation}
    \item [3.]If $ \rho >  \frac{\xi_M}{z_\varepsilon}$ , then the optimal solution to \eqref{1intervention} is 
    \begin{equation}
          \hat{t}= T - \frac{z_\varepsilon}{\xi_M}, \qquad \hat{\xi}_\varepsilon = \xi_M.
                    \label{3}
    \end{equation}
    \end{itemize}

  \end{theorem}
 \begin{proof}
By choosing a suitable target set $L$ and threshold $\varepsilon$ that satisfy \eqref{cond}, we
have naturally ruled out $t=T$. As we can move $t$ arbitrary close to $T$, from \eqref{quantile},  we obtain a natural upper constraint on the time $t$,
\begin{equation*}
    t \le T - \frac{z_\varepsilon}{\xi_M}.
\end{equation*}
Considering the minimization problem \eqref{1intervention} with the constraint \eqref{xi} we get first order condition for an interior optimum as
\begin{equation*}
    \frac{1}{T-  \hat{t}} = \rho,
\end{equation*}
or 
\begin{equation*}
 \hat{t}= T-\frac{1}{\rho},
\end{equation*}
which is feasible whenever 
\begin{equation*}
    \rho \ge \frac{1}{T} \quad \& \quad \rho \le  \frac{\xi_M}{z_\varepsilon}.
\end{equation*}
In this case, we install $ \hat{\xi}_\varepsilon \coloneqq \rho z_\varepsilon$ capacity which verifies naturally the constraint $\hat{\xi}_\varepsilon < \xi_M$. This concludes the proof for \eqref{1}. 

If $ \rho < \frac{1}{T}$, then $ \hat{t}$ as in \eqref{1} is not feasible. By simple computations one can see that in this case, the constrained cost function increases with time, thus the optimal time is achieved at 
 $ \hat{t}= 0$ and consequently  $ \hat{\xi}_\varepsilon \coloneqq \frac{z_\varepsilon}{T}$ which is feasible by \eqref{quantile}, this proves \eqref{2}.
 
 Finally, if $\rho > \frac{\xi_M}{z_\varepsilon}$, then the optimal time $  \hat{t} = T - \frac{1}{\rho}$ is feasible, but the corresponding installation exceeds the maximum value as $\hat{\xi} = \rho z_\varepsilon > \xi_M$. In this case, as the constrained cost function decreases with time, the constraint forces us to intervene at   
 $   \hat{t} = T - \frac{z_\varepsilon}{\xi_M}$ and install the maximum capacity $ \hat{\xi}_\varepsilon = \xi_M$, as in \eqref{3}. This completes the proof. 
 \end{proof}

 \begin{remark}
      Notice that, if assumption \eqref{quantile} does not hold, then it may be possible that, for an optimal intervention $\hat{t}$, the corresponding capacity expansion $\hat{\xi}_\varepsilon \coloneqq \frac{z_\varepsilon}{T- \ \hat{t}} > \xi_M$.  This means that the capacity expansion associated to the optimal intervention time is not feasible, i.e. we cannot satisfy the emission target.
\end{remark}

\subsection{Analysis discussion}
In this part, we will discuss the key findings from our analysis  on the one-intervention case and interpret their implications. 

In Victoria \textit{et al.} \cite{VBG}, decarbonization pathways for Europe over a 30-year horizon are examined, assuming constant district heating penetration, constant annual heat demand, and fixed electricity transmission capacities after 2030. The authors found that a transition characterized by early and steady $CO_2$ reductions, particularly in the first decade, is more cost-effective than trajectories requiring more drastic reductions later. This led us to ask: Can the presence of uncertainty influence this result? We incorporated accumulated emissions over the entire time horizon and uncertainty in verifying a given emission bound in the transition to a low-carbon energy system. In this context, we find that intervening and installing new capacity can  minimize the costs associated with renewable energy installations while satisfying emission constraints. The choice of the optimal time to intervene depends on the size of the discount factor; it may be optimal to install all capacity immediately or to wait for an optimal time: fixing $T = 30$, i.e. a time horizon up to about $2055$, with a discount rate $\rho > \frac{1}{30} \approx 3\%$ (annually) it is optimal to wait $30 - 1/\rho$ years before installing renewable capacity. With a $5\%$ discount rate, we get $1/\rho=20$, so we should delay the installation by $10$ years from the start. Additionally, we need to account for $\xi_M$ and $z_\varepsilon$ in this discussion, but if we suppose they satisfy hypothesis \eqref{quantile}, then this is the case. Our conclusion is that there are cases when it is optimal to wait for the installation of renewable and this is at stake with the findings in Victoria \textit{et al.} \cite{VBG}. If we do not have an interior solution, e.g. whether $\rho< \frac{1}{30}$, then  $t=0$ is the optimal intervention time. Finally, if $\rho > \xi_M/z_\varepsilon$ then $t= T - z_\varepsilon/\xi_M$ is the optimum.

Due to the urgency of addressing climate change and the need to limit global warming, it is reasonable to consider a relatively low threshold for $L$. A commonly used benchmark in climate science is the concept of the carbon budget, which represents the maximum amount of $CO_2$ emissions that can be emitted while remaining within a given temperature goal. For example, the  Intergovernmental Panel on Climate Change (IPCC) estimated in $2018$ that to limit with medium to high confidence global warming to $1.5$ degrees Celsius above pre-industrial levels, the remaining carbon budget for the whole world is between $420$ and $580$ gigatons of $CO_2$. This budget should be divided among different countries and sectors according to their historical emissions and contributions to global warming. 

We illustrate our results in Figure \ref{fig1},  for three different cases of the discounting factor $\rho$. We choose  the emission target $L=2700$ tons of C$02$, the emission threshold $\varepsilon=0.2$, the maximum capacity $\xi_M = 1000$ GW, $e_F=0.7$ tons of C$02/$GW per year and $\kappa=1$.  The random variables for capacity factor  $V$ and energy demand $D$ follow two uniform distributions, respectively in $[0,1]$ and $[1000,1500]$ GW. The result shows that, in the first two cases, the optimal theoretical time value coincides with the numerical one within the constraint. While in the third case, when $\rho=0.08$, the optimum for the constrained optimization problem does not coincide with the minimum of the cost function. This occurs because the optimal installation linked to the minimizing time does not satisfy the emission constraint, and thus we need to intervene earlier. 
\begin{figure}[H]
    \centering
  \includegraphics[scale=0.20]{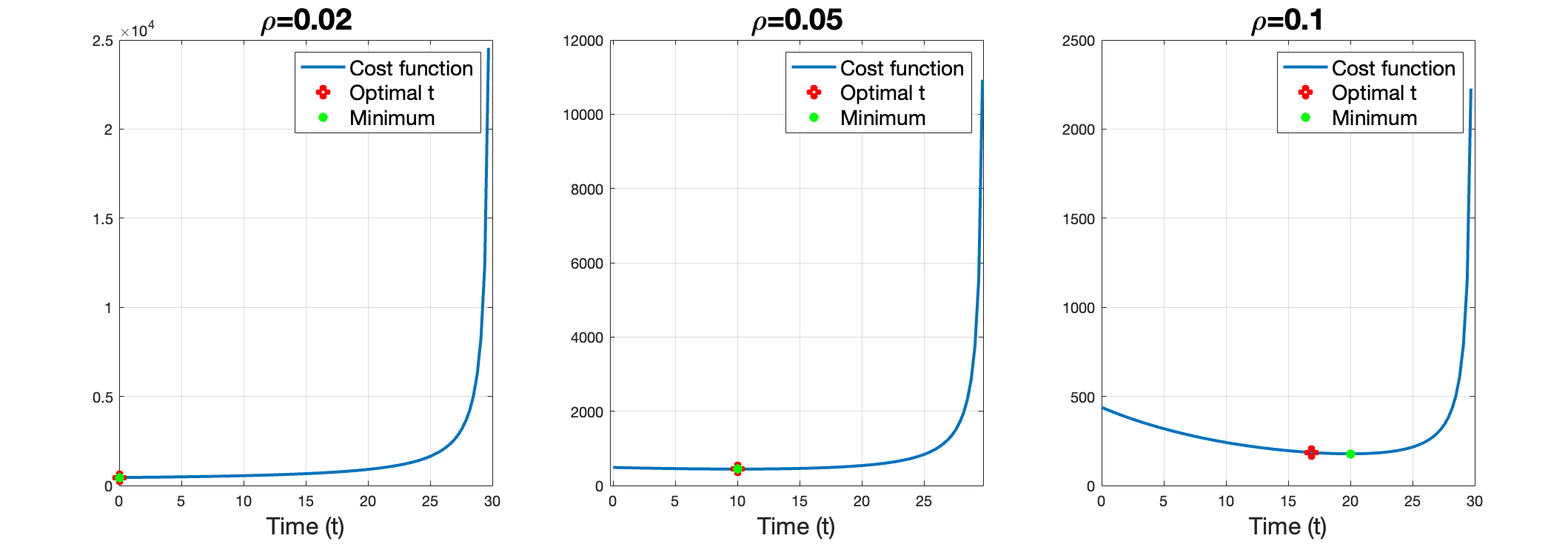}
    \caption{Cost function with the identified minimum under different discount factors}
    \label{fig1}
\end{figure}

\section{Case II: Two interventions}
In this section we will focus on the scenario involving two possible interventions. Our findings will reveal that it is never optimal to intervene more than once.

Consider $0 \le t_1 < t_2 \le T $ with associated capacity expansions $\xi_1,\xi_2 \in (0,\xi_M]$. Suppose that the initial capacity is given by $C_R(0^-)=0$, there is no fixed initial cost and the proportional cost at time $t$ is given by $K_v(t)= \kappa e^{-\rho t}$ for a discount rate $\rho >0 $ and a cost $\kappa \in \mathbb{R}^d_+$. 

Let $v=(t_1,t_2,\xi_1,\xi_2)$ be the decision variable we can control, then the function of the installed capacities is given by  $C_R(t) = \xi_1\mathbbm{1}{[t \ge t_1]} +\xi_2\mathbbm{1}{[t \ge t_2]} $ and the cost minimization problem is 
\begin{equation}
    \min_{t_i,\xi_i, i=1,2}\kappa \xi_1 e^{-\rho t_1} +\kappa \xi_2 e^{-\rho t_2} \qquad \textnormal{subject to } \qquad \mathbb{P}(E(T) > L ) = \varepsilon,
    \label{2intervention}
\end{equation}
where $L>0$ is a fixed target for the accumulated emissions and $\varepsilon > 0 $ a given threshold. 
The variable power production at time $t$ is now given by 
\begin{equation*}
    P_R(t) \coloneqq C_R(t)V(t)=(\xi_1\mathbbm{1}{[t \ge t_1]} +\xi_2\mathbbm{1}{[t \ge t_2]})V(t), 
\end{equation*}
with $C_F \in \mathbb{R}_+$ representing the maximal possible power production from fossil sources. 
Note that in this case 
\begin{equation*}
    D(t)-P_R(t) \coloneqq \begin{cases}
  D(t) & \textnormal{if } t < t_1,  \\
       D(t)-\xi_1V(t)  & \textnormal{if } t \in [t_1,t_2), \\ D(t)-\xi_1V(t) - \xi_2V(t)
   & \textnormal{if }  t\in (t_2,T].
\end{cases}
\end{equation*}
Moreover, assuming that $D(t) - P_R(t) \le C_F$, the total accumulated emissions at time $t$ corresponds to 
\begin{equation*}
\begin{aligned}
    E(T) =& 
    \int_0^{t_1}e_F D(s)ds + \int_{t_1}^{t_2}e_F(D(s)-\xi_1V(s))^+ds \\ &+ \int_{t_2}^{T}e_F(D(s)-\xi_1V(s)-\xi_2V(s))^+ds.
    \end{aligned}
\end{equation*}
Suppose again that $\forall s, \; D(s) = D, V(s) = V$ and satisfy $\bP(D < 2\xi_MV) \approx 0$ i.e. renewable energy production alone is not enough to supply the demand. Then $E$ becomes, 
\small
\begin{equation}
\begin{aligned}
    E(T)&=t_1e_F D + (t_2-t_1)e_F(D-\xi_1V )+ (T-t_2)e_F(D-\xi_1V-\xi_2V )
     \\&= e_F \big [ TD - V \big [ (T-t_1) \xi_1 + (T-t_2)\xi_2 \big ]\big ],
     \label{emissions}
    \end{aligned}
\end{equation} 
with corresponding partial derivatives with respect to $t_1$ and $t_2$ given as
\begin{equation*}
    \partial_{t_i} E = e_FV\xi_i \ge 0; \quad i=1,2.
\end{equation*}
For $\xi_1, \xi_2$ fixed, $E$ increases with these time variables. In fact, the earlier we act, the smaller emissions we will produce. 
On the other hand,
\begin{equation*}
    \partial_{\xi_i} E = -V(T-t_i) \le 0; \quad i=1,2.
\end{equation*}
Thus, for $t_1, t_2$ fixed, $E$ decreases in $\xi_i$. Clearly, the higher $\xi_i$ we choose, the less emission we will produce, but recall that the 
cost will also be higher. 

Using the emissions formula \eqref{emissions}, the constraint given in the optimization problem \eqref{2intervention} is equivalent to
 \begin{equation*}
\begin{aligned} 
    \mathbb{P} \bigg(  \frac{TD - \frac{L}{e_F}}{V} > (T - t_2)\xi_2 + (T-t_1)\xi_1 \bigg ) = \varepsilon.
\end{aligned}
\end{equation*}
Hence the quantity 
\begin{equation}
    z_\varepsilon \coloneqq(T - t_2)\xi_2 + (T-t_1)\xi_1
    \label{quant2}
\end{equation}is the $\varepsilon-$upper quantile of the random variable $\frac{TD - \frac{L}{e_F}}{V}$. Notice that, unlike the case with one intervention, for each given $t_1$ and $t_2$, we will have  an infinite number of possible capacity combinations $\xi_1, \xi_2$. \\
\subsection{Direct approach}
Here, we extend the one intervention case into two interventions and solve it analytically. 
\begin{theorem}
   For a given target set for accumulated emissions $L > 0$, and a threshold $\varepsilon > 0$ as described in \eqref{cond}, assuming that both the demand $D$ and capacity factor $V$ are time-independent random variables satisfying \eqref{DV} and such that the $\varepsilon-$upper quantile $z_\varepsilon$ satisfies \eqref{quantile}. Then, the optimal solution to  \eqref{2intervention} is to intervene once.
    \end{theorem}
    \begin{proof}
From the emission constraint \eqref{quant2}, we can write 
$$ \xi_1= \frac{1}{T-t_1}\left[ z_\varepsilon - \xi_2 (T-t_2) \right],$$
and the minimization problem becomes 
$$\min_{t_1,t_2, \xi_2} \kappa\frac{1}{T-t_1}\left[ z_\varepsilon - \xi_2 (T-t_2) \right]e^{-\rho t_1}+\kappa\xi_2e^{-\rho t_2}. $$ 
First order optimality conditions give:
\begin{itemize}
    \item For $t_1$
\begin{align}
    &\frac{\partial}{\partial t_1} 
\left[\frac{1}{T-t_1}\left[ z_\varepsilon - \xi_2 (T-t_2) \right]e^{-\rho t_1}+\xi_2e^{-\rho t_2} \right] \nonumber 
\\&\qquad= \left[ z_\varepsilon - \xi_2 (T-t_2) \right]  \frac{1}{(T-t_1)^2}e^{-\rho t_1} \left[ \frac{1}{T-t_1} - \rho \right]
 \label{1.1}  
\\&\qquad= 0,  \nonumber
\end{align}
    \item For $t_2$
\begin{align}
    &\frac{\partial}{\partial t_2} \left[\frac{1}{T-t_1}\left[ z_\varepsilon - \xi_2 (T-t_2) \right]e^{-\rho t_1}+\xi_2e^{-\rho t_2} \right]  \label{2.1}\\&\qquad= \xi_2 \left[ \frac{1}{T-t_1}e^{-\rho t_1}  - \rho e^{-\rho t_2} \right]\nonumber 
\\&\qquad= 0, \nonumber
\end{align}    
    \item For $\xi_2$ 
\begin{align}
       &\frac{\partial}{\partial \xi_2} \left[\frac{1}{T-t_1}\left[ z_\varepsilon - \xi_2 (T-t_2) \right]e^{-\rho t_1}+\xi_2e^{-\rho t_2} \right] \label{3.1} \\ &\qquad= e^{-\rho t_2} - \frac{1}{T-t_1}(T-t_2) e^{-\rho t_1} \nonumber 
\\&\qquad= 0.   \nonumber
\end{align}
\end{itemize}
Consequently, the critical points will follow.

Start looking at \eqref{2.1}, we can assume $\xi_2 \ne 0$, otherwise we are in the case with only one intervention. It follows that 
$$e^{-\rho \hat{t}_2} = \frac{1}{(T-\hat{t}_1)\rho}e^{-\rho \hat{t}_1},$$
and by substituting in \eqref{3.1} we obtain
$$\frac{1}{(T-\hat{t}_1)\rho}e^{-\rho \hat{t_1}} - \frac{1}{(T-\hat{t}_1)}(T-\hat{t}_2) e^{-\rho \hat{t}_1} = 0. 
 $$ This leads to 
 \begin{equation*}
     \frac{1}{\rho}- (T-\hat{t}_2) = 0, 
     \end{equation*}
and thus     \begin{equation*}
         \hat{t}_2= T-\frac{1}{\rho}.
 \end{equation*}
By substituting in \eqref{1.1}, the equality holds under the condition $$\hat{t}_1= T - \frac{1}{\rho},$$ which implies the only critical points correspond to $\hat{t}_1=\hat{t}_2$ and this is the case with only one intervention.    The other condition for having \eqref{1.1} equal to zero is \begin{equation*}
    z_\varepsilon = \hat{\xi}_2(T-\hat{t}_2),
\end{equation*} 
which implies
\begin{equation*}
    (T-\hat{t}_1)\hat{\xi}_1 = 0. 
\end{equation*}
We are not interested in the case $\hat{\xi}_1=0$, since it consists of the case with one intervention. While if $\hat{\xi}_1 \ne 0$, $\hat{t}_1=T$ this is not feasible with the emission constraints. So the only possible optimal solution is the one with one intervention. 

We got that the only critical point corresponds to wait until $\hat{t}= T - \frac{1}{\rho}$, whenever $\rho > \frac{1}{T}$ (to ensure $t \ge 0$) and $\rho < \frac{\xi_M}{z_\varepsilon}$ to ensure $\hat{\xi} = \rho z_\varepsilon < \xi_M$. That is the same conclusion of the case with only one intervention and the optimal solution follows from Theorem \ref{1int}.
 \end{proof}
 \newpage
We shall give an alternative approach to deal with the 'two-intervention' case. 

\subsection{Alternative approach}
One could think about trying to optimize the problem over $t_2$ and $\hat{\xi}_2$, assuming to know the values of $\hat{t}_1$ and $\hat{\xi}_1$. The problem becomes \begin{equation*}
    \min_{t_2,\xi_2} \kappa \xi_2 e^{-\rho t_2} \qquad \textnormal{s.t } \qquad \xi_2\le\xi_M, \; \;t_2 \in [\hat{t}_1,T], \;\; (T-\hat{t}_1)\hat{\xi}_1+(T-t_2)\xi_2=z_\varepsilon.
\end{equation*}
Now going back to the probability constraint \eqref{quantile} we have that
\begin{equation*}
\begin{aligned}
    \mathbb{P} \bigg(  \frac{TD - \frac{L}{e_F}}{V} - (T-\hat{t}_1)\hat{\xi}_1 > (T - t_2)\xi_2\bigg ) = \varepsilon,
\end{aligned}
\end{equation*}
and hence the quantity $\tilde{z}_\varepsilon=(T - t_2)\xi_2 $ is the $\varepsilon-$upper quantile of the random variable $$\frac{TD - \frac{L}{e_F}}{V}-(T-\hat{t}_1)\hat{\xi}_1$$. 

With a similar strategy as before we find $\hat{t}_2= T - \frac{1}{\rho}$ and $\hat{\xi}_2=\rho \tilde{z}_\varepsilon$.
 Next, we optimize the values $t_1$ and $\xi_1$, i.e., the problem becomes to minimize 
\begin{equation*}
    \min_{t_1,\xi_1} \kappa \xi_1 e^{-\rho t_1}+ \rho\tilde{z}_\varepsilon e^{-\rho(T-1/\rho)} \qquad \textnormal{s.t } \qquad \xi_1\le\xi_M, \; \;t_1 \in [0,t_2], \;\; (T-t_1)\xi_1+(T-\hat{t}_2)\hat{\xi}_2=z_\varepsilon,
\end{equation*}
that is 
\begin{equation*}
    \min_{t_1,\xi_1} \kappa \xi_1 e^{-\rho t_1} \qquad \textnormal{s.t } \qquad \xi_1\le\xi_M, \; \;t_1 \in [0,t_2], \;\; (T-t_1)\xi_1+(T-\hat{t}_2)\hat{\xi}_2=z_\varepsilon.
\end{equation*}
This has $\hat{t}_1 = T - \frac{1}{\rho}$ as solution, while the quantity $z_\varepsilon$ becomes 
\begin{equation*}
    z_\varepsilon= (T-\hat{t}_1)\hat{\xi}_1+(T-\hat{t}_2)\hat{\xi}_2= \frac{1}{\rho}\rho\tilde{z}_\varepsilon + \frac{1}{\rho}\hat{\xi}_1 \end{equation*}
    which implies \begin{equation*}
         \hat{\xi}_1= \rho (z_\varepsilon - \tilde{z}_\varepsilon ). 
\end{equation*}
Notice that $\hat{t}_1=\hat{t}_2$ and $\hat{\xi}_1 + \hat{\xi}_2 = \rho \tilde{z}_\varepsilon +\rho z_\varepsilon -\rho \tilde{z}_\varepsilon = \rho z_\varepsilon$, recovering the solution of the case with one intervention.

\subsection{Case with maximum total capacity}
    An alternative strategy might be to set a maximum total capacity, $\xi_M^{tot}$ such that $\xi_1+\xi_2=\xi^{tot}_M$ i.e. we want to find the optimal installation strategy such that in the entire operation we install as much capacity as possible. The problem is \begin{equation}
        \label{2intervention2}
   \min_{t_i, \xi_i, i=1,2} \xi_1 e^{-\rho t_1} + \xi_2e^{-\rho t_2} \quad \text{s.t} \quad \xi_1+\xi_2=\xi^{tot}_M, \quad z_\varepsilon=(T-t_2)\xi_2 + (T-t_1)\xi_1. \end{equation}
 It follows 
 \begin{equation*}
     z_\varepsilon =(T-t_2)(\xi^{tot}_M-\xi_1) + (T-t_1)\xi_1 
     = (t_2-t_1)\xi_1 + (T-t_2)\xi^{tot}_M,
     \end{equation*}
     and thus
     \begin{equation*}
         \xi_1 = \frac{1}{t_2-t_1}\left [ z_\varepsilon-(T-t_2)\xi^{tot}_M\right].
 \end{equation*}
Problem \eqref{2intervention2} becomes 
 $$\min_{t_i, \xi_i, i=1,2} \; \xi^{tot}_M e^{-\rho t_2} + \frac{1}{t_2-t_1}\left [ z_\varepsilon-(T-t_2)\xi^{tot}_M\right] \left ( e^{-\rho t_1} - e^{-\rho t_2} \right). $$
First order optimality conditions become 
\begin{itemize}
    \item For $t_1$ \begin{equation}
     \begin{aligned}
     \label{1.1.1}
         &\frac{\partial}{\partial t_1}\left[\xi_M e^{-\rho t_2} + \frac{1}{t_2-t_1}\left [ z_\varepsilon-(T-t_2)\xi^{tot}_M\right] \left ( e^{-\rho t_1} - e^{-\rho t_2} \right) \right]  \\ &\qquad= \left [ z_\varepsilon-(T-t_2)\xi^{tot}_M\right] \left [ \left ( \frac{1}{(t_2-t_1)^2}- \frac{\rho}{t_2-t_1} \right ) e^{-\rho t_1} - \frac{1}{(t_2-t_1)^2}e^{-\rho t_2} \right ]\\&\qquad=0.
     \end{aligned}
 \end{equation}
 \item For $t_2$
 \begin{equation}
     \begin{aligned}
     \label{1.2.1}
         &\frac{\partial}{\partial t_2}\left[\xi^{tot}_M e^{-\rho t_2} + \frac{1}{t_2-t_1}\left [ z_\varepsilon-(T-t_2)\xi^{tot}_M\right] \left ( e^{-\rho t_1} - e^{-\rho t_2} \right) \right] \\&\qquad= -\rho\xi^{tot}_M e^{-\rho t_2} -\frac{1}{(t_2-t_1)^2}\left [ z_\varepsilon-(T-t_2)\xi^{tot}_M\right]\left ( e^{-\rho t_1} - e^{-\rho t_2} \right)  \\
         &\qquad\qquad+ \frac{1}{t_2-t_1}\xi^{tot}_M \left ( e^{-\rho t_1} - e^{-\rho t_2} \right) \\ &\	\qquad\qquad +
         \frac{1}{t_2-t_1}\left [ z_\varepsilon-(T-t_2)\xi^{tot}_M\right](\rho e^{-\rho t_2})\\ &\qquad =0.
     \end{aligned}
 \end{equation}
 \end{itemize}
First-order optimality condition \eqref{1.1.1} holds if
\begin{equation*}
    z_\varepsilon = (T-t_2)\xi^{tot}_M,
\end{equation*}
which implies
\begin{equation*}
     \hat{\xi}_1=0 \quad \text{and} \quad  \hat{\xi}_2= \xi^{tot}_M.
\end{equation*}
I.e., we should intervene once. If instead
 $$\left ( \frac{1}{(\hat{t}_2-\hat{t}_1)^2}- \frac{\rho}{\hat{t}_2-\hat{t}_1} \right ) e^{-\rho \hat{t}_1} - \frac{1}{(\hat{t}_2-\hat{t}_1)^2}e^{-\rho \hat{t}_2}=0,  $$
 we obtain the equation
 \begin{equation*}
      e^{-\rho \hat{t}_2} = \left ( 1-\rho(\hat{t}_2-\hat{t}_1) \right) e^{-\rho \hat{t}_1}.
 \end{equation*}
 By substituting in \eqref{1.2.1}, we get
 \begin{equation*}
     \begin{aligned}
    0&=\rho^2(\hat{t}_2-\hat{t}_1)\xi^{tot}_M - \rho^2z_\varepsilon+ \rho^2(T-\hat{t}_2)\xi^{tot}_M, \\
          z_\varepsilon&=\xi^{tot}_M \left [ T -\hat{t}_2 + \hat{t}_2 - \hat{t}_1 \right ], 
     \end{aligned}
 \end{equation*}
 that is 
 \begin{equation*}
     \hat{t}_1= T-\frac{z_\varepsilon}{\xi^{tot}_M}.
 \end{equation*}
We get $\hat{t}_2$ from 
 $$ e^{-\rho \hat{t}_2} = \left ( 1-\rho(\hat{t}_2-\hat{t}_1) \right) e^{-\rho \hat{t}_1}.$$
 In particular, by considering \begin{equation*}
     g(t_2)= e^{-\rho t_2} - \left ( 1-\rho(t_2-\hat{t}_1) \right) e^{-\rho \hat{t}_1},
 \end{equation*}
 we see that $g(\hat{t}_1) = 0$ and 
 \begin{equation*}
     \begin{aligned}
         \frac{\partial}{\partial t_2} g(t_2) = -\rho e^{\rho t_2}+ \rho e^{-\rho ( T- z_\varepsilon/\xi_M) },
     \end{aligned}
 \end{equation*}
 changes its sign only once and exactly at $\hat{t}_2 = \hat{t}_1$. Thus, we have again shown that it is optimal to intervene only once.

\section{Case III: Multiple interventions}

In this section, we generalize the cases addressed above, we will show that even when it is admissible to intervene at $N$ different times, it will always be more optimal to intervene only once. 
 
Consider the deterministic times $0 \le t_1 < \ldots < t_N \le T $ with associated deterministic capacity expansions $\xi_1,\ldots,\xi_N \in (0,\xi_M]$. 
Let $v=(t_1,\ldots,t_N,\xi_1,\ldots,\xi_N)$ be the decision variable we can control. Then the function of the installed capacities is given by  $C_R(t) = \sum_{i=1}^N\xi_i\mathbbm{1}{[t \ge t_i]} $ and the cost minimization problem is 
\begin{equation}
    \min_{v\in\mathcal A)} \sum_{i=1}^N \kappa \xi_i e^{-\rho t_i} \qquad \textnormal{subject to } \qquad \mathbb{P}(E(T) > L ) = \varepsilon,
    \label{Nintervention}
\end{equation}
where $L>0$ is a fixed target for the accumulated emissions and $\varepsilon > 0 $ a given threshold.

Supposing that $\forall s, \; D(s) = D, V(s) = V$ and also that $\mathbb{P}(D < N\xi_M V) \approx 0$, i.e., $D-N\xi_MV> 0 $ $a.s.$ 
This means
that renewable energy alone will not be enough to supply the demand. Then the accumulated emissions up to the final time are given by:
\begin{equation*}
\begin{aligned}
    E(T)&= e_F \left( TD - V \sum_{i=1}^N (T-t_i) \xi_i \right).
    \end{aligned}
\end{equation*}
\begin{theorem}
  For a given target set for accumulated emissions $L > 0$, and a threshold $\varepsilon > 0$ as described in \eqref{cond}, assuming that both demand $D$ and capacity factor $V$ are time-independent random variables satisfying \eqref{DV} and such that the $\varepsilon-$upper quantile $z_\varepsilon$ satisfies \eqref{quantile}. Then, the optimal solution to \eqref{Nintervention} is to intervene once.
    \end{theorem}
    \begin{proof}
The problem now is
 $$\min_{t_i, \xi_i \; i=1,\ldots, n} \sum_{i=1}^n \xi_i e^{-\rho t_i} \quad \text{s.t.} \quad z_\varepsilon = \sum_{i=1}^n \xi_i (T-t_i).$$ 
 Thus, we can write 
 $$ \xi_1= \frac{1}{T-t_1} \left [ z_\varepsilon - \sum_{i=2}^n \xi_i (T-t_i) \right],$$
 and the problem becomes 
 $$ \min_{t_i,\xi_i \; i=1,\ldots, n}    \frac{1}{T-t_1} \left [ z_\varepsilon - \sum_{i=2}^n \xi_i (T-t_i) \right]e^{-\rho t_1} + \sum_{i=2}^n \xi_i e^{-\rho t_i}.$$ 
 The first order optimality conditions become,
 \begin{itemize}
    \item For $t_1$
 \begin{equation*}
\begin{aligned}
    &\frac{\partial}{\partial t_1} \left[\frac{1}{T-t_1} \left [ z_\varepsilon - \sum_{i=2}^n \xi_i (T-t_i) \right]e^{-\rho t_1} + \sum_{i=2}^n \xi_i e^{-\rho t_i} \right] \\&\qquad= \left[ z_\varepsilon - \sum_{i=2}^n \xi_i (T-t_i)   \right] \left[ \frac{1}{(T-t_1)^2}e^{-\rho t_1} - \frac{\rho}{T-t_1}e^{-\rho t_1}\right] \\&\qquad= 0. 
    \end{aligned}
\end{equation*}
\item  For $t_i$
\begin{equation}
\begin{aligned}
    &\frac{\partial}{\partial t_i} \left[\frac{1}{T-t_1} \left [ z_\varepsilon - \sum_{i=2}^n \xi_i (T-t_i) \right]e^{-\rho t_1} + \sum_{i=2}^n \xi_i e^{-\rho t_i} \right]\\ &\qquad= \frac{1}{T-t_1}\xi_ie^{-\rho t_1} - \rho\xi_ie^{-\rho t_i} = 0; \quad \forall \; i=2,\ldots,n. \label{5}
    \end{aligned}
\end{equation}
    \item For $\xi_i$
\begin{equation}
\begin{aligned}
    &\frac{\partial}{\partial \xi_i} \left[\frac{1}{T-t_1} \left [ z_\varepsilon - \sum_{i=2}^n \xi_i (T-t_i) \right]e^{-\rho t_1} + \sum_{i=2}^n \xi_i e^{-\rho t_i} \right] \\ &\qquad= e^{-\rho t_i} - \frac{1}{T-t_1}(T-t_i) e^{-\rho t_1} = 0;  \quad \forall \; i=2,\ldots,n, \label{6}
    \end{aligned}
\end{equation}
\end{itemize}
 From \eqref{6}, we have
 $$e^{-\rho \hat{t}_i} =  \frac{1}{T-\hat{t}_1}(T-\hat{t}_i) e^{-\rho \hat{t}_1}, $$ and substituting in \eqref{5}
 $$ \frac{1}{T-\hat{t}_1}\xi_ie^{-\rho \hat{t}_1} - \rho \xi_i \frac{1}{T-\hat{t}_1}(T-\hat{t}_i) e^{-\rho \hat{t}_1}  = 0,$$
 we find
 $$ 1-\rho(T-\hat{t}_i) = 0 \quad \Rightarrow \quad \hat{t}_i = T -  \frac{1}{\rho} \quad \forall \; i = 2, \ldots, n.$$
 We lead back to the case with two interventions previously addressed, with optimality reached by one intervention only.
    \end{proof}
\bigskip


\subsection{Conclusions}
Our main findings show that, under uncertainty, delaying green investments might be advantageous. This outcome contrasts with the conclusions of Victoria \textit{et al.} \cite{VBG}. Additionally, we determine that a single investment is more effective than two or multiple ones. This perspective differs from the stepwise installation of renewables commonly planned by many authorities. 

For future research, we could explore using random intervention times instead of predetermined ones. Additionally, investigating stochastic dynamics for both demand and capacity, as well as examining multidimensional scenarios involving multiple technologies and locations, could provide further insights.

It may also be of interest to note that from the perspective of investors, the "all at once" approach for renewable energy installation may raise concerns. Such a strategy could lead to a sudden surge in demand for materials, labor and other resources, with possible supply chain disruptions and cost increases.



\begin{thebibliography}{99}

\bibitem{ADT} A\"id, R., Dumitrescu, R., \& Tankov, P. (2021). The entry and exit game in the electricity markets: A mean-field game approach. Journal of Dynamics \& Games, 8(4):331--358.

\bibitem{BDT} Bassi\'ere, A., Dumitrescu, R., \& Tankov, P. (2024). A mean-field game model of electricity market dynamics. In Quantitative Energy Finance: Recent Trends and Developments (pp. 181-219). Cham: Springer Nature Switzerland.

\bibitem{CD} Calafiore, G., \& Dabbene, F. (Eds.). (2006). Probabilistic and randomized methods for design under uncertainty. Springer Science \& Business Media.


\bibitem{D} Dentcheva, D. (2006). Optimization models with probabilistic constraints. In Springer Science \& Business Media (Ed.) Probabilistic and randomized methods for design under uncertainty, (pp. 49--97). 

\bibitem{DSBR} D{\'\i}az-Gonz{\'a}lez, F., Sumper, A., Gomis-Bellmunt, O., \& Villaf{\'a}fila-Robles, R. (2012). A review of energy storage technologies for wind power applications. Renewable and Sustainable Energy Reviews, 16(4), 2154--2171.

\bibitem{DLT} Dumitrescu, R., Leutscher, M., \& Tankov, P. (2024). Energy transition under scenario uncertainty: a mean-field game of stopping with common noise. Mathematics and Financial Economics, {\it available online, to appear}.

\bibitem{Groch} Grochowicz, A.,  van Greevenbroek, K., Benth, F. E., and Zeyringer, M. (2023). Intersecting near-optimal spaces: European power systems with more resilience to weather variability. Energy Economics, 118:106496.

\bibitem{GrochBloom} Grochowicz, A.,  van Greevenbroek, K., and Bloomfield, H. (2024). Using power system modelling outputs to identify weather-induced extreme events in highly renewable systems. Environmental Research Letters, 19:054038. 

\bibitem{MBKB} Mohseni, S., Brent, A. C., Kelly, S., \& Browne, W. N. (2022). Demand response-integrated investment and operational planning of renewable and sustainable energy systems considering forecast uncertainties: A systematic review. Renewable and Sustainable Energy Reviews, 158:112095.

\bibitem{PDPKD} Pfeifer, A., Dobravec, V., Pavlinek, L., Kraja{\v{c}}i{\'c}, G., \& Dui{\'c}, N. (2018). Integration of renewable energy and demand response technologies in interconnected energy systems. Energy, 161:447--455.

\bibitem{RJWLL} Ren, H., Jiang, Z., Wu, Q., Li, Q., \& Lv, H. (2023). Optimal planning of an economic and resilient district integrated energy system considering renewable energy uncertainty and demand response under natural disasters. Energy, 277:127644.

\bibitem{Ro} Ross, S. M. (1996). Stochastic Processes (2nd. Ed.). New York: John Wiley \& Sons.

\bibitem{Ru} Rudin, W. (1964). Principles of Mathematical Analysis. Cocos (Keeling) Islands: McGraw-Hill.

\bibitem{VBG} Victoria, M., Zhu, K., Brown, T., Andersen, G. B., and Greiner, M. (2020). Early decarbonisation of the European energy system pays off. Nature Communications, 11:6223.

\bibitem{Zeyr} Zeyringer, M., Price, J., Fais, B., Li, P.-H., and Sharp, E. (2018). Designing low-carbon power
systems for Great Britain in 2050 that are robust to the spatiotemporal and inter-annual variability of
weather. Nature Energy, 3(5):395--403.

\end{thebibliography}

\end{document}